\documentclass[11pt]{amsart}
\usepackage{amssymb,amscd}
\begin{document}
\newtheorem{thm}{Theorem}[section]
\newtheorem{lem}[thm]{Lemma}
\newtheorem{prop}[thm]{Proposition}
\newtheorem{cor}[thm]{Corollary}
\theoremstyle{definition}
\newtheorem{ex}[thm]{Example}
\newtheorem{rem}[thm]{Remark}
\newtheorem{prob}[thm]{Problem}
\newtheorem{thmA}{Theorem}
\renewcommand{\thethmA}{}
\newtheorem{defi}[thm]{Definition}
\renewcommand{\thedefi}{}
\input amssym.def
\long\def\alert#1{\smallskip{\hskip\parindent\vrule%
\vbox{\advance\hsize-2\parindent\hrule\smallskip\parindent.4\parindent%
\narrower\noindent#1\smallskip\hrule}\vrule\hfill}\smallskip}
\def\ff{\frak}
\def\Spec{\mbox{\rm Spec}}
\def\type{\mbox{ type}}
\def\Hom{\mbox{ Hom}}
\def\rank{\mbox{ rank}}
\def\Ext{\mbox{ Ext}}
\def\Ker{\mbox{ Ker}}
\def\Max{\mbox{\rm Max}}
\def\End{\mbox{\rm End}}
\def\l{\langle\:}
\def\r{\:\rangle}
\def\Rad{\mbox{\rm Rad}}
\def\Zar{\mbox{\rm Zar}}
\def\Supp{\mbox{\rm Supp}}
\def\Rep{\mbox{\rm Rep}}
\def\cal{\mathcal}
\title[Profinite MV/BL-algebras]{Profinite MV-algebras}
\thanks{\today}
\author{Jean B Nganou}
\address{Department of Mathematics, University of Oregon, Eugene,
OR 97403} \email{nganou@uoregon.edu}
\begin{abstract} We characterize all profinite MV-algebras, these are MV-algebras that are limits of finite MV-algebras. It is shown that these are exactly direct product of finite \L ukasiewicz's chains. We also prove that the category $\mathbb{M}$ of multisets is dually equivalent to the category $\mathbb{P}$ of profinite MV-algebras and homomorphisms that reflect principal maximal ideals. Thus generalizing the corresponding result for finite MV-algebras, and finite multisets.
 \vspace{0.20in}\\
{\noindent} Key words: MV-algebra, profinite, multiset, dually equivalent, maximal ideal, finitely approximable.
\end{abstract}
\maketitle
\section{Introduction}
MV-algebras were introduced by Chang in order to provide an algebraic proof of the completeness theorem of \L ukasiewicz many-valued logic \cite{CC}. \\
An MV-algebra is an Abelian monoid $(A,\oplus , 0)$ with an involution $\neg:A\to A$ (i. e.; $\neg \neg x=x$ for all $x\in A$) satisfying the following axioms for all $x, y\in A$: $\neg 0\oplus x=\neg0$;
$\neg(\neg x\oplus y)\oplus y=\neg(\neg y\oplus x)\oplus x$. For any $x,y\in A$, we write $x\leq y$ if $\neg x\oplus y=\neg 0:=1$. Then, $\leq$ induces a partial order on $A$, which is in fact a lattice order where $x\vee y=\neg(\neg x\oplus y)\oplus y$ and $x\wedge y=\neg(\neg x\vee \neg y)$. An ideal of an MV-algebra is a nonempty subset $I$ of $A$ such that (i) for all $x,y\in I$, $x\oplus y\in I$ and (ii) for all $x\in A$ and $y\in I$ with $x\leq y$, then $x\in I$. A prime ideal of $A$ is proper ideal $P$ such that $x\wedge y\in P$, then $x\in P$ or $y\in P$. Maximal ideal has the usual meaning.\\
The prototype of MV-algebra is the unit real interval $[0,1]$ equipped with the operation of truncated addition $x\oplus y=\text{max}\{x+y,1\}$, negation $\neg x=1-x$, and the identity element $0$. For each integer $n\geq 2$, $\L_n=\left\{0, \dfrac{1}{n-1}, \cdots, \dfrac{n-2}{n-1}, 1\right\}$ is a sub-MV-algebra of $[01]$ (the \L ukasiewicz's chain with $n$ elements), and up to isomorphism every finite MV-chain is of this form.\\
The concept of profiniteness has been investigated on several classes of algebras of logic. It is well known (see, e.g., \cite[Sec.VI.2 and VI.3]{Jo}) that a Boolean algebra is profinite if and only if it is complete and atomic, that a distributive lattice is profinite if and only if it is complete and completely join-prime generated \cite[Thm. 4.4]{GG}, and that a Heyting algebra is profinite if and only if it is finitely approximable, complete and completely join-prime generated \cite[Thm. 3.6]{GG}.  Some other notable works on profinite algebras include: profinite topological orthomodular lattices \cite{CG}, profinite completions of Heyting algebras \cite{GGMM}, and profinite MV-spaces  \cite{ng, dng}.\\
There is no known simple description of the dual space of MV-algebras comparable to Esakia space for Heyting algebras, or Prietley spaces for distributive lattices. For this reason, we carry a completely algebraic analysis of profinite MV-algebras. We obtain that an MV-algebra is profinite if and only if it is isomorphic to a direct product of finite MV-chains. It follows that an MV-chain is profinite if and only if it is finite.\\
It is well known that the category $\mathbb{FMV}$ of finite MV-algebras is dually equivalent to the category of finite multisets. Among the categories that contain $\mathbb{FMV}$ as a full subcategory, one has the category $\mathbb{LFMV}$ of locally finite MV-algebras, and the category $\mathbb{P}$ of profinite MV-algebras. The duality was extended to locally finite MV-algebras in \cite{CDM}, yielding an equivalence between generalized multisets and $\mathbb{LFMV}^{op}$. Very recently, the later duality was extended further to locally weakly finite MV-algebras \cite{CM}. In the last section of the paper, we extend the duality to profinite MV-algebras, and obtain that the category $\mathbb{M}$ of multisets is dually equivalent to the category $\mathbb{P}$ of profinite MV-algebras and homomorphisms that reflect principal maximal ideals.
\section{Profinite MV-algebras}
Recall that an inverse system in a category $\mathcal{C}$ is a family $\{A_i,\varphi_{ij}\}_{i\in I}$ of objects, indexed by a directed poset $I$ (for every $i, j\in I$, there exists $k\in I$ such that $i\leq k$ and $j\leq k$), together with a family of morphisms $\varphi_{ij}:A_j\to A_i$, or each $i\leq j$, satisfying the following conditions. \\
$(i) \varphi_{kj}=\varphi_{ki}\circ \varphi_{ij}$ for all $k\leq i\leq j$;\\
$(ii) \varphi_{ii}=1_{A_i}$ for all $i\in I$.\\
Given an inverse system $\{A_i,\varphi_{ij}\}_{i\in I}$, an inverse limit of this system is an object $A$ together with a family of morphisms $\varphi_i:A\to A_i$ satisfying the condition $\varphi_{ij}\circ \varphi_j=\varphi_i$ when $i\leq j$ and having the following universal property: for every object $B$ of $\mathcal{C}$ together with a family $\psi_i:B\to A_i$, if $\varphi_{ij}\circ \psi_j=\psi_i$ for $i\leq j$, then there exists a unique morphism $\psi:B\to A$ such that $\varphi_i\circ \psi=\psi_i$ for all $i\in I$.\\
The inverse limit of an inverse system $\{A_i,\varphi_{ij}\}_{i\in I}$, when it exists, is unique up to isomorphism and often denoted by $\varprojlim \{A_i,\varphi_{ij}\}_{i\in I}$, or simply by $\varprojlim \{A_i\}_{i\in I}$ if the transition maps $(\varphi_{ij})$ are understood.\\
Using the terminology of \cite{Jo}, we call an algebra profinite if is isomorphic to the inverse limit of finite algebras of the same type.\\ Let $\left ( \{A_i,\varphi_{ij}\}\right )_I$ be an inverse system of MV-algebras. As in many varieties of algebras, it is easy to see that $$\varprojlim \{A_i\}_{i\in I}\cong \left\{(a_i)\in \prod_I A_i:\varphi_{ij}(a_j)=a_i \; \text{whenever}\; i\leq j\right\}$$
Let $A$ be an algebra and $I$ is the set of all congruences $\theta$ of $A$ such that $A/\theta$ is finite. If the class of an element $a\in A$ modulo $\theta$ is denoted by $[a]_{\theta}$, then for $\phi \subseteq \theta$, there is a canonical projection $\varphi_{\phi \theta}:A/\phi\to A/\theta$ given by $\varphi_{\phi \theta}([a]_\phi)=[a]_\theta$. It follows easily that $\left ( \{A/\theta,\varphi_{\phi \theta}\}\right )_I$ is an inverse system. Let $\widehat{A}$ be the inverse limit of this system. Then, it is well-known that $$\widehat{A}=\left\{([a]_\theta)_\theta\in \prod_IA/\theta:\varphi_{\phi \theta}([a]_\phi)=[a]_\theta \; \text{whenever}\; \phi \subseteq \theta\right\} $$
$\widehat{A}$ is called the profinite completion of the algebra $A$. Note that there is a canonical homomorphism $e:A\to \widehat{A}$ given by $e(a)=([a]_\theta)_{\theta\in I}$.\\
To start our algebraic analysis of profinite MV-algebras, we find necessary conditions to profiniteness. One clear such condition is finite approximability. Recall that an algebra is called finitely approximable if it is (isomorphic) to a sub-algebra of a direct product of finite algebras of the same type. It is known (se, e.g., \cite[Prop. 3.2]{GG}) that an algebra is finitely approximable if and only if the morphism $e:A\to \widehat{A}$ above is injective. For MV-algebras, given that every finite MV-algebras is the direct product of finite MV-chains, finite approximability means isomorphic to a sub-MV-algebra of a direct product of finite MV-chains.
\begin{prop}\label{complete}Let $A$ be a profinite MV-algebra. Then,\\
(1) $A$ is complete and finitely approximable.\\
(2) $A$ is simple if and only if $A$ a finite MV-chain.\\
\end{prop}
\begin{proof} Suppose $A=\varinjlim A_i\cong \{(a_i)\in \prod_I A_i:\varphi_{ij}(a_j)=a_i \; \text{whenever}\; i\leq j\}$, where each $A_i$ is a finite MV-algebra.\\
(1) Let $S$ be a nonempty subset of $A$, then $S$ is a subset of $\prod_I A_i$, which is complete as each of the $A_i$ is. For each $i\in I$, let $S_i$ denotes the projection of $S$ onto $A_i$. Then each $S_i$ is finite and in $\prod_I A_i$, $\vee S=(\vee S_i)_{i\in I}$. Since the transition morphisms preserve finite suprema, then it is readily verified that $\vee S\in A$. On the other hand, it is clear from the definitions that any profinite MV-algebra is finitely approximable.\\
(2) If $A$ is simple, every homomorphism with domain $A$ is injective. For each $i$, there is a projection $p_i:A\to A_i$ that must be one-to-one, and forcing $A$ to be finite since $A_i$ is. But, finite simple MV-algebras are MV-chains. It is also known that finite MV-chains are simple.
\end{proof}
The following result can be derived from the theory of vector lattices, but for the convenience of the reader, we give a direct self-contained proof here.
\begin{lem}\label{notpro}
Let $A$ be a direct product of copies of $[0,1]$, that is $A=[0,1]^{X}$ for some nonempty set $X$, and $M$ be a maximal ideal of $A$. Then,  $$A/M\cong [0,1]$$
\end{lem}
\begin{proof}
For any MV-algebra $A$, let $\mathcal{H}(A)$ denotes the set of MV-algebra homomorphisms from $A$ into $[0,1]$ and $Max(A)$ denotes the set of maximal ideals of $A$. It is well known \cite{CDM} that $\chi\mapsto ker\chi$ defines a one-to-one correspondence between $\mathcal{H}(A)$ and $Max(A)$, where $ker \chi=\{a\in A: \chi(a)=0\}$. As a consequence, if $A$ is simple, there is a unique (injective) homomorphism from $A\to [0,1]$; in particular $\mathcal{H}([0,1])=\{Id\}$. Now, suppose $A=[0,1]^{X}$ and $M\in Max(A)$, then $M=ker \chi$ for some $\chi\in \mathcal{H}(A)$. We need to justify that $A/ker \chi$ is isomorphic to $[0,1]$. Consider the map $\tau :[0,1]\to A$ defined by $\tau (t)(x)=t$ for all $t\in [0,1]$ and $x\in X$, then $\tau$ is clearly a homomorphism. Thus, $\chi \circ \tau$ is a homomorphism $[0,1]\to [0,1]$ and it follows from  $\mathcal{H}([0,1])=\{Id\}$ that $\chi \circ \tau=Id$. Now, consider the map $\theta: [0,1]\to A/ker \chi$ defined by $\theta (t)=\tau (t)/ker \chi$, in other words $\theta$ is the composition of $\tau$ followed by the natural projection $A\to A/ker \chi$. Then $\theta$ is a homomorphism, and we claim that $\theta$ is an isomorphism. Since $[0,1]$ is simple, then $\theta$ is injective. For the surjectivity, let $f\in A$ and $t=\chi (f)$. Then, since $\chi \circ \tau=Id$,  $(\chi \circ \tau)(t)=t$. So, $\neg ((\chi \circ \tau )(t))\otimes \chi (f)=0$ and $\neg (\chi (f))\otimes (\chi \circ \tau)(t)=0$. Therefore, $\neg f\otimes \tau (t), \neg (\tau (t))\otimes f\in ker \chi$ and $f/ker \chi =\tau(t)/ker \chi=\theta(t)$. Hence, $\theta$ is an isomorphism as claimed.
\end{proof}
\begin{prop}\label{pro01}
For every non-empty set $X$, the MV-algebra $[0,1]^X$ is not finitely approximable.
\end{prop}
\begin{proof}
Let $A=[0,1]^X$ and suppose by contradiction that $A$ is finitely approximable, then there is a homomorphism from $A$ into a finite MV-algebra (any of the projections). Since every finite MV-algebra is a product of finite MV-chains \cite[Prop. 3.6.5]{C2}, then there exists an integer $n\geq 2$ and a homomorphism $p:A\to \L_n$. Thus, $A/ker p$ is isomorphic to a subalgebra of $\L_n$ and therefore by \cite[Thm. 3.5.1]{C2}, $A/ker p$ is simple, from which it follows that $ker p$ is a maximal ideal of $A$. But, $A/ker p$ is infinite by Lemma \ref{notpro} , which contradicts the fact that $A/ker p$ is isomorphic to a subalgebra of $\L_n$.
\end{proof}
\begin{prop}
The MV-algebra $[0,1]\times \mathbf{2}$ (where $\mathbf{2}$ is the 2-element Boolean algebra) is not finitely approximable.
\end{prop}
\begin{proof}
By contradiction suppose that $[0,1]\times \mathbf{2}$ is finitely approximable. Then there exists finite MV-chains $\L_{n_i}$; $i\in I$ and a one-to-one homomorphism $\tau:[0,1]\times \mathbf{2} \to \prod_{i\in I}\L_{n_i}$. For each $i\in I$, let $p_i$ denotes the natural projection $\prod_{i\in I}\L_{n_i} \to \L_{n_i}$, and consider $\phi_i=p_i\circ \tau$. Then each $\phi_i$ is a homomorphism from $[0,1]\times \mathbf{2}\to [0,1]$ and it follows that $ker \phi_i$ is a maximal ideal of $[0,1]\times \mathbf{2}$. But, $[0,1]\times \mathbf{2}$ has exactly two maximal ideals: $[0,1]\times \{0\}$ and $\{0\}\times \mathbf{2}$. Note that it is not possible to have $ker \phi_i=\{0\}\times \mathbf{2}$, for this would imply by the homomorphism theorem that $Im \phi_i \cong [0,1]\times \mathbf{2}/ker \phi_i\cong [0,1]$. And this would contradict the fact that $Im \phi_i$ is finite as it is a sub-MV-algebra of $\L_{n_i}$. Therefore, $ker \phi_i=[0,1]\times \{0\}$ for all $i\in I$. Thus, for every $t\in [0,1]$, and every $i\in I$, $\phi_i(t,0)=0$, that is $p_i(\tau(t,0))=0$. Hence, $\tau(t,0)=0$ and $t=0$, which is contradiction. Hence, $[0,1]\times \mathbf{2}$ is not finitely approximable as claimed.
\end{proof}
\begin{cor}\label{pro02}
If $A$ is the direct product of MV-chains, among which $[0,1]$, then $A$ is not finitely approximable.
\end{cor}
\begin{proof}
Every such MV-algebra contains a sub-MV-algebra isomorphic to $[0,1]\times \mathbf{2}$. In fact, suppose that $A=\prod_{i\in I}C_i$, where $C_{i_0}=[0,1]$ for some $i_0\in I$. Let $S_{i_0}=\{f\in A:f(i)=0, \text{for all}\;  i\ne i_0\}\cup\{f\in A:f(i)=1, \text{for all}\;  i\ne i_0\}$. Then $S_{i_0}$ is a sub-MV-algebra of $A$, that is clearly isomorphic to $[0,1]\times \mathbf{2}$.
\end{proof}
The next result offers a simple algebraic characterizations of profinite MV-algebras.
\begin{thm}\label{MV-case}
For every non-trivial MV-algebra $A$, the following assertions are equivalent:\\
 (1) $A$ is profinite\\
 (2) $A$ is complete and finitely approximable\\
 (3) $A$ is isomorphic to the direct product of finite MV-chains.\\
\end{thm}
\begin{proof}
(1)$\Rightarrow$ (2): Obvious.\\
(2)$\Rightarrow$ (3):  Suppose that $A$ is a sub-MV-algebra of $\prod_I A_i$, where $A_i$ is a finite MV-algebra. Since each $A_i$ is a finite MV-algebra, then it is isomorphic to a (finite) product of finite MV-chains. So, $\prod_I A_i$ is isomorphic to a direct product of finite MV-chains, and by \cite[Thm. 6.8.1]{C2}, $\prod_I A_i$ is complete and completely distributive. Since $A$ is a complete sub-MV-algebra of $\prod_I A_i$, then $A$ is completely distributive. Therefore, by \cite[Thm. 6.8.1]{C2} again, $A$ is a direct product of complete MV-chains. But, every complete MV-chain is isomorphic to a finite MV-chain, or to $[0,1]$. Moreover, since $A$ is finitely approximable, it follows from Corollary \ref{pro02} that $A$ is a direct product of finite MV-chains. \\
(3)$\Rightarrow$ (1): Is clear.
\end{proof}
It follows that profinite MV-chains are finite.
\begin{cor}\label{MV-chain}
A (non-trivial) MV-chain $A$ is profinite if and only if $A$ is isomorphic to $\L_n$ for some $n\geq 2$.
\end{cor}
\section{Maximal ideals of profinite MV-algebras}
For any MV-algebra $A$, let $\mathcal{H}(A)$ denotes the set of MV-algebra homomorphisms from $A$ into $[0,1]$ and $Max(A)$ denotes the set of maximal ideals of $A$. It is well known \cite{CDM} that $\chi\mapsto ker\chi$ defines a one-to-one correspondence between $\mathcal{H}(A)$ and $Max(A)$, where $ker \chi=\{a\in A: \chi(a)=0\}$. We will use the following notations through out the paper. For $A:=\displaystyle\prod_{x\in X}\L_{n_x}$ a profinite MV-algebra, and each $x\in X$, $p_x:A\to \L_{n_x}$ denotes the natural projection. In addition $M_x$ will denote $ker p_x$, it follows that each $M_x$ is a maximal ideal of $A$. It is easy to see that $\displaystyle\oplus_{x\in X}\L_{n_x}:=\left\{f\in A:f(x)=0\; \text{for all, but finitely many}\;  x\in X \right \}$ is an ideal of $A$. Recall that a principal ideal of an MV-algebra $A$ is any ideal $I$ that is generated by a single element, that is there exists $a\in A$, such that $I=\langle a\rangle$. It is well known that $x\in \langle a\rangle$ if and only if $x\leq na$ for some integer $n\geq 1$.
\begin{lem}\label{principal}
Let $A:=\displaystyle\prod_{x\in X}\L_{n_x}$ be a profinite MV-algebra. For any maximal ideal $M$ of $A$, the following conditions are equivalent.\\
(i) $M$ is principal;\\
(ii) There exists a unique $x_0\in I$, such that $M=ker p_{x_0}$;\\
(iii) $M$ does not contain $\oplus_{x\in X}\L_{n_x}$.
\end{lem}
\begin{proof}
$(i)\Rightarrow (ii)$: Suppose that $M$ is principal, then $M=\langle a\rangle $ for some $a\in A$. We claim that there exists $x_0\in X$ with $a(x_0)=0$. By contradiction suppose that $a(x)\ne 0$ for all $x\in X$, then for each $x\in X$, $a(x)=\dfrac{k_{n_x}}{n_x-1}$ for some $1\leq k_{n_x}\leq n_x-1$.\\
We consider two cases:\\
(a) $\left\{\dfrac{n_x-1}{k_{n_x}}\right\}_{x\in X}$ is bounded, then there exits an integer $m\geq 1$ such that $\dfrac{n_x-1}{k_{n_x}}\leq m$ for all $x\in X$. It follows that $ma=1$, and so $M=A$, which is a contradiction.\\
(b) $\left\{\dfrac{n_x-1}{k_{n_x}}\right\}_{x\in X}$ is unbounded. Then $\left\{n_x\right\}_{x\in X}$ is unbounded. We can write $X$ as the disjoint union of of two sets $X'$ and $X''$ such that $\left\{n_x\right\}_{x\in X'}$ and $\left\{n_x\right\}_{x\in X''}$ are unbounded. Define $f,g\in A$ by:\\
$f(x)=
\left\{\begin{array}{ll}
  1 & ,\ \ \mbox{if} \ \ x\in X'\\
 a(x)&  ,\ \ \mbox{if} \ \ x\in X''
\end{array}\right.$
and \ \ \ 
$g(x)=
\left\{\begin{array}{ll}
  a(x) & ,\ \ \mbox{if} \ \ x\in X'\\
 1&  ,\ \ \mbox{if} \ \ x\in X''
\end{array}\right.$\\
Then $f\wedge g=a$, in particular $f\wedge g\in M$. Since $M$ is prime, as every maximal ideal is, then $f\in M$ or $g\in M$. Assume $f\in M=\langle a\rangle$, then there exists an integer $r\geq 1$ such that $f\leq ra$. Therefore, $1\leq \dfrac{rk_{n_x}}{n_x-1}$ for all $x\in X"$, and so $\dfrac{n_x-1}{k_{n_x}}\leq r$ for all $x\in X"$. This contradicts the fact that $\left\{n_x\right\}_{x\in X''}$ is unbounded. In a similar argument, $g\in M$ would contradict the fact that $\left\{n_x\right\}_{x\in X'}$ is unbounded.\\
Thus $a(x_0)=0$ for some $x_0\in X$. For every $f\in M=\langle a\rangle$, there exists $k\geq 1$ such that $f\leq ka$, and it follows that $f(x_0)=0$ for all $f\in M$. Hence, $M\subseteq M_{x_0}$. Since $M$ and $M_{x_0}$ are maximal, then $M=M_{x_0}=ker p_{x_0}$. The uniqueness is clear.\\
$(ii)\Rightarrow (i)$ This is clear as each $M_{x_0}$ is principal as it is generated by $f(x)=
\left\{\begin{array}{ll}
  0 & ,\ \ \mbox{if} \ \ x=x_0\\
 1&  ,\ \ \mbox{if} \ \ x\ne x_0
\end{array}\right.$\\

$(ii) \Rightarrow (iii)$: Suppose that there exists a unique $x_0\in I$, such that $M=ker p_{x_0}$. Consider $f\in A$ defined by $f(x)=
\left\{\begin{array}{ll}
  1 & ,\ \ \mbox{if} \ \ x=x_0\\
 0&  ,\ \ \mbox{if} \ \ x\ne x_0
\end{array}\right.$
Then $f\in \displaystyle\oplus_{x\in X}\L_{n_x}$ and $f\notin M$.\\
$(iii)\Rightarrow (ii)$ Suppose that for all $x\in X$, $M\ne M_x$. For each $x\in X$, let $b_x\in A$ defined by $b_x(t)=
\left\{\begin{array}{ll}
  0 & ,\ \ \mbox{if} \ \ t=x\\
 1&  ,\ \ \mbox{if} \ \ t\ne x
\end{array}\right.$\\
Then for every $x\in X$, since $M_x=\langle b_x\rangle$, then $b_x\notin M$ and since $M$ is maximal, by \cite[Prop.1.2.2]{C2} there exists an integer $k_x\geq 1$ such that $\neg k_xb_x=\neg b_x\in M$. It follows that $M$ contains $\displaystyle\oplus_{x\in X}\L_{n_x}$.
\end{proof}
\begin{cor}
Let $A:=\displaystyle\prod_{x\in X}\L_{n_x}$ be a profinite MV-algebra. A maximal ideal $M$ of $A$ is not principal if and only if $\displaystyle\oplus_{x\in X}\L_{n_x} \subseteq M$.
\end{cor}
\section{A Stone type duality}
We say that a homomorphism $\varphi:A\to B$ of MV-algebras reflect principal ideals if for every principal ideal $J$ of $B$, $\varphi^{-1}(J)$ is a principal ideal of $A$. It is clear that the identity reflects principal maximal ideals, and that the composition of two homomorphisms that reflect principal maximal ideals also reflects principal maximal ideals. Let  $\mathbb{P}$ denotes the category of profinite MV-algebras and homomorphisms that reflect principal maximal ideals. We recall the definition of the category $\mathbb{M}$ of multisets. A multiset is a pair $\langle X, \sigma :X\to \mathbb{N}\rangle$, where $X$ is a set and $\sigma$ is a map. Given two multisets $\langle X, \sigma \rangle$ and $\langle Y, \mu \rangle$, a morphism from $\langle X, \sigma \rangle$ to $\langle Y, \mu \rangle$ is a map $\varphi: X\to Y$ such that $\mu(\varphi(x))$ divides $\sigma(x)$ for all $x\in X$.\\
We shall define two functors $\mathcal{H}: \mathbb{P}^{op}\to \mathbb{M}$ and $\mathcal{F}: \mathbb{M}\to \mathbb{P}^{op}$
\begin{itemize}
\item[(1)] $\mathcal{H}: \mathbb{P}^{op}\to \mathbb{M}$. For any profinite MV-algebra $A$, set $$\mathcal{H}_F(A):=\left\{\chi:A\to [0,1]:\chi \; \text{is a homomorphism and}\; ker\chi \; \text{is principal (maximal) ideal}\right\}$$ and $\sigma_A:\mathcal{H}_F(A)\to \mathbb{N}$ defined by $\sigma_A(\chi)=\#\chi(A)-1$. \\ Note that $\sigma_A$ is well-defined because as $ker\chi$ is principal, by Lemma \ref{principal} and the homomorphism theorem, $\chi(A)$ is finite.
\begin{itemize}
\item On objects: Given a profinite MV-algebra $A$, define $\mathcal{H}(A)=\langle \mathcal{H}_F(A), \sigma_A\rangle$.
\item On morphisms: let $\varphi$ be a homomorphism in $\mathbb{P}^{op}$ from $A\to B$, that is $\varphi:B\to A$ is an MV-algebra homomorphism that reflects principal ideals. Define $\mathcal{H}(\varphi): \mathcal{H}_F(A) \to \mathcal{H}_F(B)$ by $\mathcal{H}(\varphi)(\chi)=\chi \circ \varphi$. Note that since $ker\chi(A)$ is a principal maximal ideal of $A$, $ker(\chi \circ \varphi)=\varphi^{-1}(ker\chi)$, and $\varphi$ reflects principal maximal ideals, then $ker(\chi \circ \varphi)$ is principal maximal. So, $\chi \circ \varphi\in \mathcal{H}_F(B)$ and $\mathcal{H}(\varphi)$ is well-defined. On the other hand, note by \cite[Cor. 3.5.4, Cor. 7.2.6]{C2} that for each $\chi\in \mathcal{H}_F(A)$, $\chi(A)=\L_{\#\chi(A)}$. Thus, $\L_{\#(\chi \circ \varphi)(A)}\subseteq \L_{\#\chi(A)}$, and it follows that $\#(\chi \circ \varphi)(A)-1$ divides $\#\chi(A)-1$. Thus, $\sigma_B(\mathcal{H}(\varphi)(\chi))$ divides $\sigma_A(\chi)$ for all $\chi\in \mathcal{H}_F(A)$. Therefore, $\mathcal{H}(\varphi)$ is a morphism in $\mathbb{M}$ from $\mathcal{H}_F(A) \to \mathcal{H}_F(B)$.\\
\end{itemize}
\item[(2)] $\mathcal{F}: \mathbb{M}\to \mathbb{P}^{op}$. For any multiset $\langle X, \sigma\rangle$, $\displaystyle\prod_{x\in X}\L_{\sigma(x)+1}$ is clearly a profinite MV-algebra, that shall be denoted by $A_{X,\sigma}$.
\begin{itemize}
\item On objects: Given a multiset $\langle X, \sigma\rangle$, define $\mathcal{F}(\langle X, \sigma\rangle):=A_{X,\sigma}$.
\item On morphisms: Let $\varphi: \langle X, \sigma\rangle \to \langle Y, \mu\rangle$ be a morphism in $\mathbb{M}$. Define $\mathcal{F}(\varphi):A_{Y,\mu}\to A_{X,\sigma}$ by $\mathcal{F}(\varphi)(f)(x)=f(\varphi(x))$ for all $f\in A_{Y,\mu}$ and all $x\in X$. To see that $\mathcal{F}(\varphi)$ is well-defined, first note that for all $f\in A_{Y,\mu}$ and all $x\in X$, $f(\varphi(x))\in \L_{\mu(\varphi(x))+1}$. On the other hand, $\mu(\varphi(x))$ divides $\sigma(x)$, hence $ \L_{\mu(\varphi(x))+1}\subseteq \L_{\sigma{x}+1}$. Thus, $f(\varphi(x))\in \L_{\sigma{x}+1}$. In addition, let $M$ be a principal maximal ideal of $A_{X,\sigma}$, then by Lemma \ref{principal}, there exists $x_0\in X$ such that $M=M_{x_0}$. It is easy to see that $\mathcal{H}(\varphi)^{-1}(M_{x_0})=M_{\varphi(x_0)}$, which is a principal maximal ideal of $A_{Y,\mu}$. Finally, it is easy to see that $\mathcal{H}(\varphi)$ is a MV-homomorphism from $A_{Y,\mu} \to A_{X,\sigma}$.\\
\end{itemize}
\end{itemize}
The only missing aspects of the proof of the following results are simple computations, which we shall omit. 
\begin{prop}
$\mathcal{H}: \mathbb{P}^{op}\to \mathbb{M}$ and $\mathcal{F}: \mathbb{M}\to \mathbb{P}^{op}$ are functors.
\end{prop} 
\begin{prop}
Let $\langle X, \sigma \rangle$ be a multiset, define $\eta_X:\langle X, \sigma \rangle \to \langle \mathcal{H}_F(A_{X,\sigma}),\sigma_{A_{X,\sigma}}\rangle$ by $\eta_X(x)(f)=f(x)$, for all $x\in X$ and all $f\in A_{X,\sigma}$.\\
Then $\eta_X$ is an isomorphism in $\mathbb{M}$.
\end{prop}
\begin{proof}
Note that for each $x\in X$, $\eta_X(x)$ is a homomorphism from $A_{X,\sigma}\to \L_{\sigma(x)+1}$, in particular $\eta_X(x)\in \mathcal{H}_F(A_{X,\sigma})$ and $\eta_X$ is well-defined. To see that $\eta_X$ is a morphism, let $x\in X$, then $\eta_X(x)(A_{X,\sigma})\subseteq \L_{\sigma(x)+1}$. Thus, $\L_{\#\eta_X(x)(A_{X,\sigma})}\subseteq \L_{\sigma(x)+1}$, hence $\#\eta_X(x)(A_{X,\sigma})-1$ divides $\sigma(x)$. Whence, $\sigma_{A_{X,\sigma}}(\eta_X(x))$ divides $\sigma(x)$ for all $x\in X$.\\
 It remains to prove that $\eta_X$ is bijective.\\
Injectivity: Let $x_1,x_2\in X$ such that $x_1\ne x_2$. Define $f\in A_{X,\sigma}$ by $f(x_1)=0$ and $f(x)=1$ for $x\ne x_1$. Then $\eta_X(x_1)(f)=0$, while $\eta_X(x_2)(f)=1$. Therefore $\eta_X(x_1)\ne \eta_X(x_2)$ and $\eta_X$ is injective.\\
Surjectivity: Let $\chi \in \mathcal{H}_F(A_{X,\sigma})$, then $ker\chi$ is a principal maximal ideal of $A_{X,\sigma}$. By Lemma \ref{principal}, there exists $x\in X$ such that $ker \chi=M_x=kerp_x$. Hence, $\chi=p_x$, and it follows that $\eta_X(x)=\chi$.\\
Thus, $\eta_X$ is an isomorphism in $\mathbb{M}$.
\end{proof}
\begin{prop}
Let $A$ be a profinite MV-algebra. Define $\varepsilon_A: A\to \displaystyle\prod_{\chi\in \mathcal{H}_F(A)}\L_{\#\chi(A)}$ by $\varepsilon_A(f)(\chi)=\chi(f)$
for all $f\in A$ and all $\chi\in  \mathcal{H}_F(A)$.\\
Then $\varepsilon_A$ is an isomorphism in $\mathbb{P}^{op}$.
\end{prop}
\begin{proof}
Since $\chi(A)=\L_{\#\chi(A)}$ for all $\in \mathcal{H}_F(A)$, it follows that $\varepsilon_A$ is well-defined. In addition, let $M$ be a principal maximal ideal of $\displaystyle\prod_{\chi\in \mathcal{H}_F(A)}\L_{\#\chi(A)}$, then by Lemma \ref{principal}, there exists $\chi_0\in \mathcal{H}_F(A)$ such that $M=M_{\chi_0}$. But, it is clear that $\varepsilon_A^{-1}(M_{\chi_0})=ker \chi_0$, which is principal maximal ideal of $A$. Thus, $\varepsilon_A$ reflects principal maximal ideals. It is straightforward to verify that $\varepsilon_A$ is a homomorphism of MV-algebras. It remains to prove that $\varepsilon_A$ is bijective.\\
Injectivity: Let $f, g\in A$ such that $\varepsilon_A(f)=\varepsilon_A(g)$, then for all $\chi\in  \mathcal{H}_F(A)$, $\chi(f)=\chi(g)$. Since $A$ is profinite, by Theorem \ref{MV-case}, there exists a set $X$ and a sequence of integers $\{n_x\}_{x\in X}$ such that $A=\displaystyle\prod_{x\in X}\L_{n_x}$. We have $p_x(f)=p_x(g)$ for all $x\in X$, hence $f(x)=g(x)$ for all $x\in X$ and $f=g$.\\
Surjectivity: Let $g\in \displaystyle\prod_{\chi\in \mathcal{H}_F(A)}\L_{\#\chi(A)}$. Since $A$ is profinite, by Theorem \ref{MV-case}, there exists a set $X$ and a sequence of integers $\{n_x\}_{x\in X}$ such that $A=\displaystyle\prod_{x\in X}\L_{n_x}$. Then, by Lemma \ref{principal}, $x \leftrightarrow p_x$ is a one-t-one correspondence between $X$ and $\mathcal{H}_F(A)$. Now define $f\in A$ by $f(x)=g(p_x)$. Then, it follows clearly that $\varepsilon_A(f)=g$.\\
Thus, $\varepsilon_A$ is an isomorphism in $\mathbb{P}^{op}$.
\end{proof}
\begin{thm}\label{eq1}
The composite $\mathcal{H}\circ \mathcal{F}$ is naturally equivalent to the identity functor of exists a natural isomorphism $\mathbb{M}$. In other words, for all multisets $\langle X, \sigma\rangle$, $\langle Y, \mu\rangle$ and $\varphi: \langle X, \sigma\rangle \to \langle Y, \mu\rangle$ a morphism in $\mathbb{M}$, we have a commutative diagram
$$
\begin{CD}
\langle X, \sigma\rangle @>\varphi>> \langle Y, \mu\rangle\\
@V\eta_XVV @VV\eta_YV\\
\mathcal{H}(\mathcal{F}(\langle X, \sigma\rangle)) @>\mathcal{H}(\mathcal{F}(\varphi))>> \mathcal{H}(\mathcal{F}( \langle Y, \mu\rangle))
\end{CD}
$$
in the sense that, for each $x\in X$, $\mathcal{H}(\mathcal{F}(\varphi)(\eta_X(x))=\eta_Y(\varphi(x))$
\end{thm}
\begin{proof}
Let $x\in X$, then $\mathcal{H}(\mathcal{F}(\varphi))(\eta_X(x))=\eta_X(x)\circ \mathcal{F}(\varphi)$. For every $g\in A_{Y,\mu}$, 
$$
\begin{aligned}
(\eta_X(x)\circ \mathcal{F}(\varphi))(g)&=\eta_X(x)(\mathcal{F}(\varphi)(g))\\
&=\mathcal{F}(\varphi)(g)(x)\\
&=g(\varphi(x))\\
&=\eta_Y(\varphi(x))(g)
\end{aligned}
$$
Hence $\mathcal{H}(\mathcal{F}(\varphi)(\eta_X(x))=\eta_Y(\varphi(x))$ for all $x\in X$ as claimed.
\end{proof}
\begin{thm}\label{eq2}
The composite $\mathcal{F}\circ \mathcal{H}$ is naturally equivalent to the identity functor of exists a natural isomorphism $\mathbb{P}^{op}$. In other words, for all all profinite MV-algebras $A, B$ and $\varphi: A \to B$ a homomorphism in $\mathbb{P}^{op}$, we have a commutative diagram
$$
\begin{CD}
B @>\varphi>> A\\
@V\varepsilon_BVV @VV\varepsilon_AV\\
\mathcal{F}(\mathcal{H}(B)) @>\mathcal{F}(\mathcal{H}(\varphi))>> \mathcal{F}(\mathcal{H}(A))
\end{CD}
$$
in the sense that, for each $f\in B$, $\mathcal{F}(\mathcal{H}(\varphi))(\varepsilon_B(f))=\varepsilon_A(\varphi(f))$
\end{thm}
\begin{proof}
Let $f\in B$ and $\chi \in \mathcal{H}_F(A)$, then
$$
\begin{aligned}
\mathcal{F}(\mathcal{H}(\varphi))(\varepsilon_B(f))(\chi)&=\varepsilon_B(f)(\mathcal{H}(\varphi)(\chi))\\
&=\varepsilon_B(f)(\chi\circ \varphi)\\
&=(\chi\circ \varphi)(f)\\
&=\chi(\varphi(f))\\
&=\varepsilon_A(\varphi(f))(\chi)
\end{aligned}
$$
Hence, $\mathcal{F}(\mathcal{H}(\varphi))(\varepsilon_B(f))=\varepsilon_A(\varphi(f))$ for all $f\in B$, as desired.
\end{proof}
Combining Theorem \ref{eq1} and Theorem \ref{eq2}, we obtain the long sought duality.
\begin{cor}
The category $\mathbb{M}$ of multisets is dually equivalent to the category $\mathbb{P}$ of profinite MV-algebras and homomorphisms that reflect principal maximal ideals.
\end{cor}
\begin{rem}
While every MV-homomorphism reflects maximal ideals \cite[Prop. 1.2.16]{C2}, MV-homomorphism may not reflect principal maximal ideals. For instance, consider the simplest infinite profinite MV-algebra, namely $A=\mathbf{2}^X$ for some fixed infinite set $X$. Then $\oplus_{X}\mathbf{2}$ is an ideal of $A$, and is contained in a maximal ideal $M$ of $A$, which is not principal by Lemma \ref{principal}. But, $A/M$ is a Boolean algebra that is isomorphic to a Boolean subalgebra of $[0,1]$. Hence, $A/M\cong \mathbf{2}$. Now consider the natural projection  $p:A\to \mathbf{2}$, then $p^{-1}(0)=M$, which is not principal.
\end{rem}
\begin{rem}
The category $\mathbb{FMV}$ of finite MV-algebras is a full subcategory of $\mathbb{P}$ and when restricted $\mathbb{FMV}$, the equivalence yields the well known duality between finite MV-algebras and finite multisets.
\end{rem}

\end{document}